\documentclass{article}
\usepackage{amsmath}
\usepackage{amsthm}
\usepackage{amssymb}
\usepackage{tikz}
\usepackage{bm}
\usepackage{verbatim}
\usepackage{asymptote}
\usepackage{mathtools}
\usepackage[margin=1in]{geometry}
\usetikzlibrary{matrix}
\usepackage[english]{babel}
\usepackage[utf8]{inputenc}

\usepackage{hyperref}
\hypersetup{
    colorlinks=true,
    linkcolor=blue,
    filecolor=magenta,      
    urlcolor=cyan,
}
 
\newcommand{\R}{\mathbb{R}}

\newcommand{\C}{\mathcal{C}}
\newcommand{\D}{\mathcal{D}}

\newcommand{\E}{\mathbb{E}}
\renewcommand{\P}{\mathbb{P}}

\newcommand{\I}{\mathcal{I}}
\newcommand{\N}{\mathcal{N}}
\newcommand{\Z}{\mathbb{Z}}
\newcommand{\NN}{\mathbb{N}}
\newcommand{\tr}{\mathrm{tr}\,}
\newcommand{\cov}{\mathrm{Cov}}
\newcommand{\var}{\mathrm{Var}}
\newcommand{\ld}{\mathrm{lead}}
\newcommand{\ldeg}{\mathrm{ldeg}}
\newcommand{\e}{\varepsilon}

\renewcommand{\vec}[1]{\mathbf{#1}}

\newtheorem{theorem}{Theorem}[section]
\newtheorem{lemma}[theorem]{Lemma}
\newtheorem{proposition}[theorem]{Proposition}
\newtheorem{corollary}[theorem]{Corollary}

\theoremstyle{definition}

\theoremstyle{definition}

\title{Discrete Derivative Asymptotics of the $\beta$-Hermite Eigenvalues}
\author{Gopal Goel and Andrew Ahn}

\date{\today}

\begin{document}
\maketitle

\begin{abstract}
We consider the asymptotics of the difference between the empirical measures of the $\beta$-Hermite tridiagonal matrix and its minor. We prove that this difference has a deterministic limit and Gaussian fluctuations. Through a correspondence between measures and continual Young diagrams, this deterministic limit is identified with the Vershik-Kerov-Logan-Shepp curve. Moreover, the Gaussian fluctuations are identified with a sectional derivative of the Gaussian free field.
\end{abstract}

\section{Introduction}
For $\beta > 0$, the $\beta$-Hermite ensemble is the random point process $\lambda_1 > \cdots > \lambda_N$ with probability distribution proportional to
\begin{align} \label{eq:beta}
\prod_{1 \le i < j \le N} (\lambda_i- \lambda_j)^\beta \prod_{i=1}^N e^{-\frac{\lambda_i^2}{2}}.
\end{align}
This is the joint eigenvalue distribution of the Gaussian Orthogonal Ensemble (GOE) for $\beta = 1$, Gaussian Unitary Ensemble (GUE) for $\beta = 2$, and Gaussian Symplectic Ensemble (GSE) for $\beta = 4$; see \cite[Sections 2.5 and 4.1]{GUIONNET} for background on these classical matrix ensembles. Consider the random symmetric tridiagonal matrix
\begin{equation}\label{eq:tridiagonal}
 X_N \sim  \frac{1}{\sqrt{2}} \left( \begin{array}{ccccc} \N(0, 2) & \chi_{(N-1) \beta} & & & \\  \chi_{(N-1) \beta} & \N(0, 2)  & \chi_{(N-2) \beta} & & \\  & \ddots & \ddots & \ddots & \\ & & \chi_{2\beta} & \N(0,2) & \chi_{\beta} \\ & &  & \chi_{\beta} & \N(0,2) \end{array} \right)
\end{equation}
where we interpret $\chi_a$ to be a chi distributed random variable with parameter $a$, $\N(a,b)$ as Gaussian with mean $a$ and variance $b$, and the entries are independent random variables (except for the symmetry constraint). In \cite{EDELMAN}, Dumitriu and Edelman showed that the random symmetric tridiagonal matrix (\ref{eq:tridiagonal}) has joint eigenvalue distribution (\ref{eq:beta}) for arbitrary $\beta > 0$. For $\beta = 1,2,4$, these tridiagonal matrix models correspond to tridiagonalizations of the GOE, GUE, GSE respectively; a procedure which preserves the joint distribution of eigenvalues of the original matrix \emph{and} its minor (see Section \ref{tridiagonal_discussion}). Let $X_{N-1}$ denote the lower right $(N-1)\times(N-1)$ minor of $X_N$. Let the eigenvalues of $X_N$ and $X_{N-1}$ be denoted by $\lambda_1^{(N)} > \cdots > \lambda_N^{(N)}$ and $\lambda_1^{(N-1)} > \cdots > \lambda_{N-1}^{(N-1)}$ respectively.

In this article, we focus on the asymptotics of the difference of empirical measures
\begin{align} \label{eq:diff}
\sum_{i=1}^N \delta_{\lambda_i^{(N)}} - \sum_{i=1}^{N-1} \delta_{\lambda_i^{(N-1)}}.
\end{align}
The measure above can be interpreted as the second derivative of a continual Young diagram, a connection which is described more precisely in Section \ref{setup}. For $X_N$ a Wigner matrix, the limit of this random Young diagram as $N\to\infty$ was studied in \cite{ALEXEY} and \cite{ERDOS}. In particular, it was shown in \cite{ALEXEY} that the random Young diagram, under proper rescaling, converges to the Vershik-Kerov-Logan-Shepp curve
\[ \Omega(x) = \left\{ \begin{array}{ll}
|x| & \mbox{if $|x| > 2$} \\
\frac{2}{\pi} \left( x \arcsin \frac{x}{2} + \sqrt{4 - x^2} \right) & \mbox{otherwise.}
\end{array} \right. \]
The fluctuations from this deterministic limit were studied in \cite{ERDOS} and were identified with a sectional derivative of the $2$-dimensional Gaussian free field (GFF). The Vershik-Kerov-Logan-Shepp curve is also found to arise in asymptotic representation theory, as the limit of properly rescaled Young diagrams under the Plancherel measure \cite{KEROV}.

The appearance of a sectional derivative of the GFF is no coincidence. In \cite{BORODIN}, the random process formed by the eigenvalues of a Wigner matrix and its minors was shown to converge to the GFF; similar results exist for Wishart matrices \cite{DUPA} and $\beta$-Jacobi ensembles \cite{BOGO}. Since the measure (\ref{eq:diff}) is a discrete derivative in the direction of levels of minors, the convergence of (\ref{eq:diff}) to a sectional derivative of the GFF shows that the convergence of Wigner matrices to the GFF also holds in the derivative sense. We discuss this in more detail in Section \ref{GFF}, based off a similar discussion in \cite{ERDOS}.

The aim of this article is to extend these global asymptotic theorems to the $\beta$-Hermite tridiagonal matrices for $\beta > 0$ and to demonstrate the accessibility of these results through simple combinatorics of tridiagonal matrices. Although the theorems for $\beta = 1$ and $2$ are special cases of the results of \cite{ALEXEY} and \cite{ERDOS}, this article is the first treatment of the global asymptotics of (\ref{eq:diff}) for general $\beta > 0$ Hermite ensembles. Our theorems show that the global asymptotics of (\ref{eq:diff}) for $X_N$ distributed as (\ref{eq:tridiagonal}) depend on $\beta > 0$ only up to a multiplicative factor. This dependence on $\beta > 0$ is typical in the study of global asymptotics of $\beta$-ensembles, e.g. \cite{BOGO} and \cite{GORIN}. We note that the simplicity of our approach is a consequence of the Gaussianity in our model. In contrast, \cite{ALEXEY} and \cite{ERDOS} deal with real and complex Wigner matrices which may be non-Gaussian.

The asymptotics of (\ref{eq:diff}) were studied for a closely related model called the $\beta$-Jacobi ensemble in \cite{GORIN} through a different method using Macdonald difference operators. At the finite level, the $\beta$-Hermite ensemble can be realized as a degeneration of the $\beta$-Jacobi ensemble. Thus the limits obtained for the $\beta$-Hermite ensemble can be viewed as degenerations of the limits obtained for the $\beta$-Jacobi ensemble. However, this connection should be viewed as informal because a rigorous proof that the limit commutes with this degeneration requires more work and does not appear in the literature.

A similar model is studied in \cite{SODIN} where $\{\lambda_i^{(N-1)}\}_{i=1}^{N-1}$ are taken to be the critical points of the characteristic polynomial. The resulting difference of empirical measures also converges after proper rescaling to the Vershik-Kerov-Logan-Shepp curve. See \cite{SODIN} for comparison of the fluctuations between these two models.

The paper is organized as follows. We first provide preliminary notions and state the main results in Section \ref{setup}. Next, the proofs of the results are provided in Section \ref{proofs}. Finally, we interpret the results and provide the connection with the derivative of the GFF in Section \ref{GFF}.

\textbf{Acknowledgments.} We would like to thank Vadim Gorin for suggesting and providing direction for this project. This material is based upon work done through the PRIMES-USA program, supported by the National Science Foundation under Grant no. DMS-1519580.

\section{Preliminary Notions and Main Results}
\label{setup}

Let $\{x_i\}$ and $\{y_i\}$ be two \textit{interlacing} sequences of real numbers, i.e.
\[x_1\le y_1\le x_2\le\cdots\le x_{N-1}\le y_{N-1}\le x_N.\]

Define $w^{\{x_i\},\{y_i\}}(x)$ to be the \textit{rectangular Young diagram} of $\{x_i\}$ and $\{y_i\}$ in the following way.

Let $z_0=\sum_{i=1}^n x_i-\sum_{i=1}^{n-1}y_i$. Then, $w^{\{x_i\},\{y_i\}}(x)$ is the unique continuous function with the following properties.
\begin{itemize}
    \item $w^{\{x_i\},\{y_i\}}(x)=|x-z_0|$ for $x\le x_1$ and $x\ge x_N$.
    \item $\frac{d}{dx}w^{\{x_i\},\{y_i\}}(x) = 1$ for $x_i< x < y_i$ and $\frac{d}{dx}w^{\{x_i\},\{y_i\}}(x)=-1$ for $y_i<x<x_{i+1}$.
\end{itemize}

\begin{figure}[ht]
    \centering
    \includegraphics[scale=0.4]{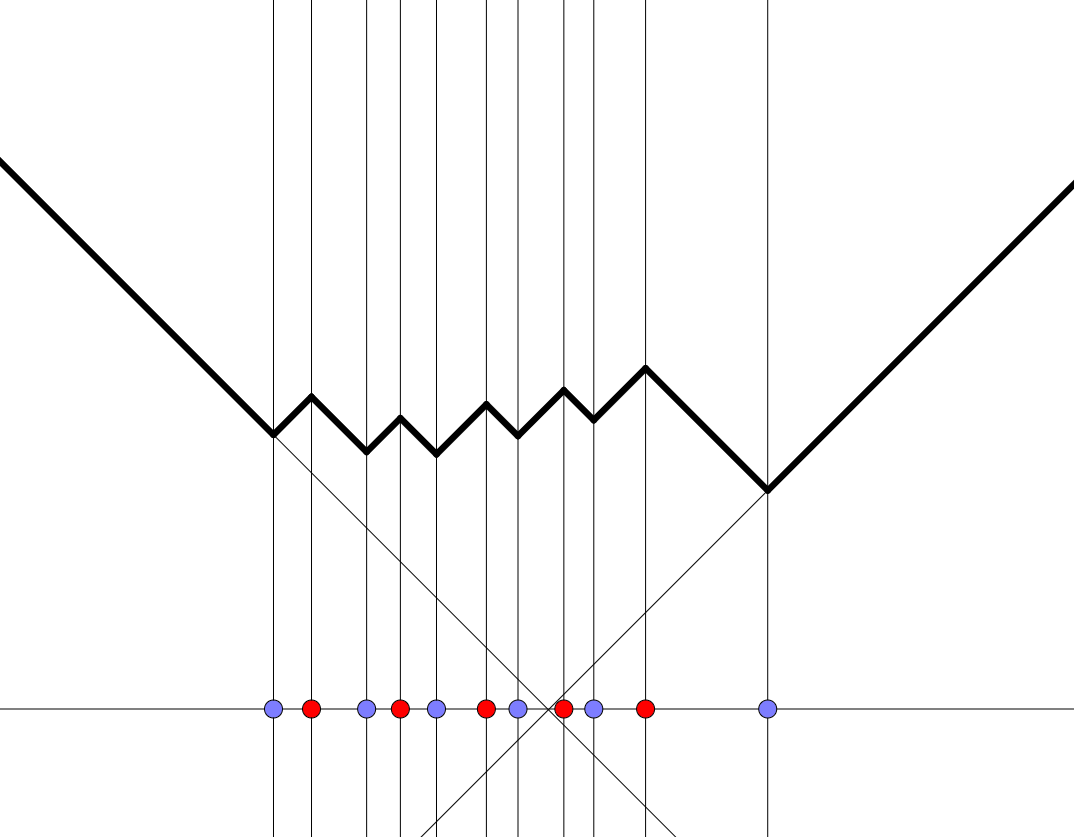}
    \caption{Here the blue points are the $x_i$ and the red points are the $y_i$.}
    \label{fig:Young}
\end{figure}
Let $A$ be an arbitrary $N$ by $N$ symmetric matrix, and let $\widehat{A}$ be its lower right $N-1$ by $N-1$ submatrix. Then, Cauchy's interlacing theorem states that the eigenvalues of $A$ and $\widehat{A}$ interlace, so we can assign a rectangular Young diagram to $A$ as in Figure \ref{fig:Young}.

Let $X_N$ be a symmetric random matrix with the distribution defined in (\ref{eq:tridiagonal}). This is the tridiagonal $\beta$-Hermite ensemble of \cite{EDELMAN}. Given $M \leq N$, the lower right $M\times M$ principal submatrix of $X_N$ is distributed as $X_M$. To preserve this dependence, when we say $X_M$ we are referring to the lower right principal submatrix $M\times M$ of $X_N$. Define $Y_M=\sqrt{\frac{2}{M\beta}}X_M$ to be a rescaling of $X_M$.

Let $w_M(x)$ be the rectangular Young diagram associated with the eigenvalues of $Y_M$ and $\widehat{Y}_M$. Consider also the random measure $\mu_M = \sum_{i=1}^M \delta_{\lambda_i^{(M)}} - \sum_{j=1}^{M-1} \delta_{\rho_j^{(M)}}$ where $\lambda_i^{(M)}$ is the $i$th eigenvalue (in some order) of $Y_M$ and $\rho_j^{(M)}$ is the $j$th eigenvalue (in some order) of $\widehat{Y}_M$. The random Young diagram is related to $\mu_M$ in the following manner
\begin{equation} \label{eq:wmu}
\frac{1}{2} (w_M(x) - |x-z_0|)'' = \mu_M - \delta_{z_0}.
\end{equation}
Let $D_{M,k}$ be the $k$th moment of $\mu_M$, or
\[D_{M,k}=\int x^k\, \mu_M(dx) = \tr Y_M^k - \tr \widehat{Y}_M^k.\]

The asymptotics of the measure $\mu_N$ are the primary focus of this article. Due to the relation above, this implies information about the convergence of the random Young diagrams. We present the results below.

\subsection{Law of Large Numbers}

\begin{theorem}[Law of Large Numbers] \label{LLN}
\[D_{N,k} \to \left\{ \begin{array}{cl} \binom{k}{k/2} & \text{if $k$ even} \\ 0 & \text{if $k$ odd} \end{array} \right.\]
in probability as $N\to\infty$.
\end{theorem}

Through (\ref{eq:wmu}), the preceding result gives information about the asymptotics of the random rectangular Young diagrams $w_N$. Let
\[\Omega(x) = \left\{
\begin{array}{cc}
    \frac{2}{\pi}(x\arcsin(\tfrac{x}{2}) + \sqrt{4-x^2}), & |x|\le 2 \\
    |x|, & |x|\ge 2
\end{array}
\right.\]
be the Vershik-Kerov-Logan-Shepp curve.

\begin{corollary} \label{dLLN} 
Let $w_N(x)$ be the random Young diagram associated to the eigenvalues of $Y_N$ and $\widehat{Y}_N$, as defined in Section \ref{setup}. Then $w_N(x) \to \Omega(x)$ as $N\to\infty$ uniformly in probability.
\end{corollary}

\subsection{Central Limit Theorem}

\begin{theorem}[Central Limit Theorem]\label{CLT}
Let $0 < \alpha_1,\ldots,\alpha_\ell \le 1$. The vector
\[ \left(\sqrt{\lfloor \alpha_i N \rfloor} ( D_{\lfloor N \alpha_i \rfloor ,k_i} - \E D_{\lfloor N \alpha_i \rfloor ,k_i} ) \right)_{i=1}^\ell \]
converges to a centered Gaussian vector $(\mathfrak{D}_{\alpha_i,k_i})_{i=1}^\ell$. The covariance structure is given by
\[ \cov(\mathfrak{D}_{\alpha_i,k_i}, \mathfrak{D}_{\alpha_j,k_j}) = \mathbf{1}_{\alpha_i = \alpha_j}\mathbf{1}_{2\mid k_i+k_j}\frac{4}{\beta}\frac{k_ik_j}{k_i+k_j} \binom{k_i + k_j - 2}{\frac{k_i+k_j-2}{2}}.  \]
\end{theorem}

We may recast Theorem \ref{CLT} in terms of fluctuations of the measure $\mu_N$. Define $\phi_N$ to be the fluctuation of $\mu_N$ given by
\[ \Phi_{N,f} := \int f(x) \phi_N(dx) = \sqrt{N} \left( \int f(x) \mu_N(dx) - \E \int f(x) \mu_N(dx) \right). \]
\begin{theorem} \label{intvar}
Let $\mathcal{P}$ be the vector space of real coefficient polynomials. Then $\{\Phi_{N,f}\}_{f \in \mathcal{P}}$ converge jointly to a centered Gaussian family $\{\Phi_f\}_{f\in\mathcal{P}}$ defined by
\[ \cov\left( \Phi_f , \Phi_g \right) = \frac{2}{\beta} \int_{-2}^2 f'(x)g'(x) \sigma(x) \, dx \]
where $\sigma(x)$ is the semicircle law.
\end{theorem}
This covariance structure can be identified with the derivative of the GFF. We leave the discussion of this identification for Section \ref{GFF}.

\section{Proofs of Results}
\label{proofs}

We set up some notation before presenting the proofs. We deal with two types of paths denoted by $\vec{i} = (i_1,\ldots,i_k)$. In one case, we will think of the indices as living in $\Z/k\Z$, that is $\vec{i} \in [N]^{\Z/k\Z}$. Later, we consider paths where the indices are in $\Z$ instead. Define
\[ \Lambda_k = \{\vec{i} = (i_1,\ldots,i_k) \in [N]^{\Z/k\Z}: |i_j - i_{j+1}| \leq 1 ~\mbox{and}~ 1 \in \vec{i} \} \]
Also let
\begin{align*}
\ell_{\vec{i}}(h) &:= \mbox{ the number of times that $\vec{i}$ hits $(h,h+1)$}, \\
m_{\vec{i}}(h) &:= \mbox{ the number of times that $\vec{i}$ hits $(h,h)$}.
\end{align*}

\subsection{Proof of Theorem \ref{LLN}}

The proof of Theorem \ref{LLN} relies on the following lemma which considers first the convergence of the expectations.

\begin{lemma}
\label{lem1}
We have that
\[\lim_{N\to\infty}\E D_{N,k} = \left\{ \begin{array}{cl} \binom{k}{k/2} & \text{if $k$ even} \\ 0 & \text{if $k$ odd.} \end{array} \right.\]
\end{lemma}

\begin{proof}
We have that
\[D_{N,k}=\tr Y_N^k - \tr Y_{N-1}^k = \sum_{\vec{i}\in\Lambda_k}\prod_{j=1}^k Y_N(i_j,i_{j+1})=:\sum_{\vec{i}\in\Lambda_k}T_{\vec{i}}.\]
We see that
\[T_{\vec{i}} = \prod_{j=1}^k Y_N(j,j)^{m_{\vec{i}}(j)}\prod_{j=1}^k Y_N(j,j+1)^{2\ell_{\vec{i}}(j)},\]
so
\[\E T_{\vec{i}} = \prod_{j=1}^k \E Y_N(j,j)^{m_{\vec{i}}(j)}\prod_{j=1}^k \E Y_N(j,j+1)^{2\ell_{\vec{i}}(j)}.\]
Note that if any of the $m_{\vec{i}}(j)$ are odd, then $\E T_{\vec{i}}=0$. For $k$ odd and for any $\vec{i} \in \Lambda_k$ there always exists an odd $m_\vec{i}(j)$, which implies $\E D_{N,k} = 0$. Let us assume $k$ is even. Then the nontrivial contributions are given by paths $\vec{i}$ for which $m_{\vec{i}}(j)$ are all even. The product of the $\E Y_N(j,j)^{m_{\vec{i}}(j)}$ is of order $N^{-\sum_j m_{\vec{i}}(j)/2}$ whereas the product of the $\E Y_N(j,j+1)^{2\ell_{\vec{i}}(j)}$ is of constant order. Thus it suffices to determine the contribution of those paths with all the $m_{\vec{i}}(j) = 0$. Let us call this set of paths
\[ \D_k = \{ \vec{i} \in \Lambda_k: |i_j - i_{j+1}| = 1 \}. \]

Since $(\beta N)^{-1} \chi_{\beta N}^2 \to 1$ almost surely, as $N\to\infty$ we have
\[\E T_{\vec{i}} = \prod_{j=1}^{k/2}\E\left(\frac{\chi_{\beta(N-j)}^2}{\beta N}\right)^{\ell_{\vec{i}}(j)} \to 1.\]
Thus
\[ \E D_{N,k} \to |\D_k| = \binom{k}{k/2}\]
as $N\to\infty$. This is true because there is a bijection between $\D_k$ and $\pm 1$ walks starting and ending at $1$ --- simply translate the path in $\D_k$ so that the path starts at $1$.
\end{proof}

We are now ready to prove Theorem \ref{LLN}.

\begin{proof}[Proof of Theorem \ref{LLN}]
By the Chebyshev inequality, for any $\e > 0$ as $N\to\infty$
\[ \P( | D_{N,k} - \E D_{N,k} | > \e ) \leq \e^{-2} \var(D_{N,k}) \to 0 \]
where we are using the fact that $\var(D_{N,k})=O(1/N)$, which we prove in the next section. This completes the proof.
\end{proof}

Given Theorem \ref{LLN}, we now present the ideas to obtain Corollary \ref{dLLN}. The ideas for this implication were mentioned in \cite{ALEXEY}. We sketch the proof by highlighting the main ideas.

We have the following fact
\begin{equation} \label{eq:lsmom}
\int_{-2}^2 \! \frac{1}{2}(\Omega(x) - |x|)''x^k\, dx = \left\{ \begin{array}{cl}
\binom{k}{k/2} & \mbox{if $k$ is even,} \\
0 & \mbox{if $k$ is odd.}
\end{array} \right.
\end{equation}

Consider also the compactly supported function
\[\kappa_N(x) = \frac{1}{2}\left(w_N(x) - |x-z_0(N)|\right) \]
where $z_0(N) = \sum_{i=1}^N \lambda_i^{(N)} - \sum_{i=1}^{N-1} \lambda_j^{(N-1)}$.

Let $\mathcal{L}([a,b])$ be the space of Lipschitz functions supported on some interval $[a,b] \subset \R$. Let the \textit{topology of moment convergence} on $\mathcal{L}([a,b])$ be the topology on which $f_n \to f$ if and only if
\[ \int f_n(x) x^k \, dx \to \int f(x) x^k\, dx \]
for $k = 0,1,2,\ldots$. We can replace $f_n$ with $f_n''$ and $f$ with $f''$ by integration by parts. From \cite{ALEXEY} (Lemma 2.1), the topology of uniform convergence on $\mathcal{L}([a,b])$ is equivalent to the topology of moment convergence.

With this equivalence of topologies, one may want to show that $\kappa_N \to \frac{1}{2}(\Omega(x) - |x|)$ in the topology of moment convergence in probability. The issue is that the support of $\kappa_N$ may be arbitrarily large. However, this is resolved because for large enough $B$, the probability that $\kappa_N$ is supported in $[-B,B]$ converges to $1$. This statement is implied directly by the following two facts. First, note that the center point of the diagram $z_0(N)$ is the $(1,1)$ entry of $Y_N$, so that $z_0(N) \to 0$ almost surely. Second, the probability that the eigenvalues of $Y_N$ are contained in $[-2-\e, 2+\e]$ (for arbitrary small $\e > 0$) converges to $1$ (e.g. see \cite{GUIONNET}, Chapter 4.5). Thus it suffices to show that
\[ \int \kappa_N'' x^k \, dx \to \int \frac{1}{2}(\Omega(x) - |x|)'' x^k \, dx \]
almost surely for $k = 0,1,2,\ldots$.

To complete the proof sketch, notice that
\[ \kappa_N'' = \sum_{i=1}^N \delta_{\lambda_i^{(N)}} - \sum_{i=1}^{N-1} \delta_{\lambda_i^{(N-1)}} + \delta_{z_0(N)}.\]
Define
\[ \mu_N = \sum_{i=1}^N \delta_{\lambda_i^{(N)}} - \sum_{i=1}^{N-1} \delta_{\lambda_i^{(N-1)}}. \]
Then
\[ \int \left(\kappa_N'' - \mu_N\right) x^k \to 0 \]
almost surely for all $k = 0,1,2,\ldots$, due again to the fact that $z_0(N) \to 0$ almost surely. The reduction is now complete because
\[ \int \mu_N x^k \, dx = D_{N,k}.\]

\subsection{Proof of Theorem \ref{CLT}}

For $k$ odd, let
\[\I_{k,h} = \left\{(i_1,\ldots,i_k) \in \Lambda_k : \exists j_0\text{ such that }(i_{j_0},i_{j_0+1})=(h,h)\text{ and }m_{\vec{i}}(h')=\delta_{h,h'}\right\},\]
let $\sigma_{k,h} = |\I_{k,h}|$, and let
\[\D_k =  \bigcup_{h\ge 1}\I_{k,h}.\]
For $k$ even, let
\[\sigma_{k,h} = \sum_{\vec{i}\in\D_k}\ell_{\vec{i}}(h).\]

\subsubsection{Preliminary Asymptotics} \label{pre}

\begin{lemma} \label{chilemma}
Let $A_N=Y_N(j,j+1)$ for some fixed nonnegative integer $j$. Then as $N\to\infty$
\begin{itemize}
    \item[(i)] $A_N \to 1$ in distribution,
    \item[(ii)] $\sqrt{N}(A_N^2 - 1)\to \mathcal{N}(0, \frac{2}{\beta})$ in distribution.
\end{itemize}
\end{lemma}
\begin{proof}
Let $\iota$ denote the imaginary unit. By Chebyshev's inequality, (ii) implies (i). The characteristic function for $A_N^2 - 1$ is given by
\[ \varphi_N(t) = \left(1 - \frac{2\iota t}{\beta N} \right)^{-\frac{1}{2} \beta (N-j)} e^{-\iota t}. \]
Then
\begin{align*}
\log \varphi_N(\sqrt{N} t) &= - \frac{1}{2} \beta (N-j) \log \left( 1 - \frac{2\iota t}{\beta \sqrt{N}} \right) - \iota \sqrt{N} t \\
&= \frac{1}{2} \beta N \left( \frac{2\iota t}{\beta \sqrt{N}} - \frac{2 t^2}{\beta^2 N} \right) - \iota \sqrt{N} t + O(N^{-1/2}) \\
&= - \frac{1}{\beta} t^2 + O(N^{-1/2}). 
\end{align*}
\end{proof}

\begin{lemma}
\label{asymp}
Let $\{(A_{N,1},\ldots,A_{N,m},B_{N,1},\ldots,B_{N,n})\}_{N=1}^\infty$ be a sequence in $N$ of random vectors with independent components. We also have $A_i \sim \frac{1}{\sqrt{\beta N}} \chi_{\beta(N-i)}$ and $B_j \sim \mathcal{N}(0,\frac{2}{\beta N})$. Let $\eta_1,\ldots,\eta_m$ and $\xi_1,\ldots,\xi_n$ be independent centered normal random variables with variance $\frac{2}{\beta}$. Then
\[ (A_{N,1},\ldots,A_{N,m},\sqrt{N}(A_{N,1}^2 - 1),\ldots,\sqrt{N}(A_{N,m}^2 - 1), \sqrt{N} B_{N,1},\ldots,\sqrt{N} B_{N,n}) \to (1,\ldots,1,\eta_1,\ldots,\eta_m,\xi_1,\ldots,\xi_n) \]
in distribution as $N\to\infty$. In particular,
\[ N^{\frac{1}{2}\left( \sum_{i=1}^m q_i + \sum_{j=1}^n r_j \right)} \E \prod_{i=1}^m A_{N,i}^{p_i} (A_{N,i}^2 - 1)^{q_i} \prod_{j=1}^n B_{N,j}^{r_j} \to \prod_{i=1}^m \E \eta_i^{q_i} \prod_{j=1}^n \E \xi_j^{r_j}. \]
\end{lemma}
\begin{proof}
It is clear from independence of all the random variables, and (ii) of lemma \ref{chilemma} that
\[ \sqrt{N} (A_{N,1}^2 - 1,\ldots,A_{N,m}^2 - 1, B_{N,1},\ldots,B_{N,n}) \to (\eta_1,\ldots,\eta_m,\xi_1,\ldots,\xi_n)\]
in distribution. We have the elementary fact that if two sequences of random variables $A_N$ and $B_N$ converge to random variables $A$ and $B$ respectively where $B$ is a constant, then $(A_N,B_N) \to (A,B)$ jointly in distribution. By (i) of lemma \ref{chilemma} and the aforementioned, this proves the desired.
\end{proof}

Define $A_{N,1},\ldots,A_{N,m},B_{N,1},\ldots,B_{N,n}$, $\eta_1,\ldots,\eta_m$, $\xi_1,\ldots,\xi_n$ as above, and let $C_{N,i} := A_{N,i}^2 - 1$. Then, let $A_N = (A_{N,1},\ldots,A_{N,m})$, $B_N = (B_{N,1},\ldots,B_{N,n})$, $C_N = (C_{N,1},\ldots,C_{N,m})$, $\boldsymbol{\eta} = (\eta_1,\ldots,\eta_m)$, $\boldsymbol{\xi} = (\xi_1,\ldots,\xi_n)$.

For our applications, the $A_{N,i}$ corresponds to the off-diagonal entry $Y(i,i+1)$ and $B_{N,j}$ corresponds to one of the diagonal entries. The importance of the previous lemmas is in identifying the order of terms. To illustrate this point, we introduce the following definitions. Let $x = (x_1,\ldots,x_m)$, $y = (y_1,\ldots,y_n)$, $z = (z_1,\ldots,z_m)$ be ordered sets of formal variables. Fix $f \in \R[x,y,z]$. We can define a sequence of random variables $\{ f(A_N,B_N,C_N) \}_{N=1}^\infty$ by evaluating $f$. If $f$ is a monomial, define
\[ \deg_{\N} f := \deg_y f + \deg_z f\]
which we will refer to as the $\N$-\textit{degree} of $f$. For general $f$, define
\[ \ldeg_{\N} f := \min \deg_{\N} \alpha \]
where the minimum is over all monomials $\alpha$ of $f$. Note that
\[\ldeg_\N fg = \ldeg_\N f + \ldeg_\N g.\]

Let $\ld(f)$ be the sum of the monomials of $f$ which have minimal $\N$-degree. Let $f_1,\ldots,f_\ell \in \R[x,y,z]$. Then
\begin{align*}
\E\prod_{i=1}^\ell f_i(A_N,B_N,C_N) &= \E\prod_{i=1}^\ell \ld(f_i)(A_N,B_N,C_N) + o(N^{-\ldeg_\N(f_1\cdots f_\ell)/2})\\
&= N^{-\ldeg_\N(f_1\cdots f_\ell)/2}\E \prod_{i=1}^\ell \ld(f_i)(1^m,\boldsymbol{\xi},\boldsymbol{\eta}) + o(N^{-\ldeg_\N(f_1\cdots f_\ell)/2})
\end{align*}
Both equalities follow from lemma \ref{asymp}.

\subsubsection{}\label{proof_CLT}

\begin{proof}[Proof of Theorem \ref{CLT}]
Define
\[ \Delta_{M,k} = \sqrt{M}( D_{M,k} - \E D_{M,k}). \]
Note that $\Delta_{M,k}$ is only dependent on the entries $(i,j)$ of $Y_M$ where $i,j$ is bounded by some constant dependent only on $k$. Let $S$ be the set of $\alpha_i$. For each $\alpha \in S$, we may create the subvector $\Delta_\alpha := (\Delta_{\lfloor \alpha_i N \rfloor,k_i})_{i: \alpha_i = \alpha}$. By the aforementioned, the collections $\{\Delta_\alpha\}_{\alpha \in S}$ are mutually independent for large $N$. Therefore it suffices to consider $\alpha_1 = \cdots = \alpha_\ell = 1$.

Let $A_{N,i} = Y_N(i,i+1)$, $B_{N,i} = Y_N(i,i)$, and $C_{N,i}=Y_N(i,i+1)^2-1$. Fix some large constant $K$ (dependent only on $k_1,\ldots,k_\ell$), in our case it suffices to choose $K = \max(k_1,\ldots,k_\ell)+1$. The elements of the collection
\[\mathcal{R}_N := \{A_{N,i}\}_{i=1}^K \cup \{B_{N,i}\}_{i=1}^K \cup \{C_{N,i}\}_{i=1}^K\]
are of the form of the identically named random variables from Section \ref{pre}. Define $\boldsymbol{\xi}$ and $\boldsymbol{\eta}$ as in Section \ref{pre}.

We show joint Gaussianity of $\sqrt{N}(\Delta_{N,k_i})_{i=1}^\ell$ by showing convergence of the joint moments to that of the appropriate Gaussian joint moments. In particular, we look at
\[\E\prod_{i=1}^\ell (\sqrt{N}\Delta_{N,k_i})^{j_i}\]
for some vector $(j_1,\ldots,j_\ell)$ of nonnegative integers. To do this, we first write the $\Delta_{N,k_i}$'s as polynomials in $\mathcal{R}_N$. Then, by the discussion in Section \ref{pre}, it suffices to just consider the leads of the $\Delta_{N,k_i}$'s evaluated at $\left(1^K,\frac{1}{\sqrt{N}}\boldsymbol{\xi},\frac{1}{\sqrt{N}}\boldsymbol{\eta}\right)$ when computing the joint moments.

First consider $k$ odd with $k < K$. Then
\[ \sqrt{N}\Delta_{N,k} = \sqrt{N}D_{N,k} = \sum_{\vec{i} \in \Lambda_k} \sqrt{N}T_{\vec{i}} =: F_k(\mathcal{R}_N)\]
for some fixed polynomial $F_k$ which is the sum of monomials $f_\vec{i}$
\[ f_\vec{i}(\mathcal{R}_N) := \prod_{j=1}^k B_{N,j}^{m_{\vec{i}}(j)}\prod_{j=1}^k A_{N,j}^{2\ell_{\vec{i}}(j)} \]
so that
\[ T_{\vec{i}} = \prod_{j=1}^k Y_N(i_j,i_{j+1}) = f_{\vec{i}}(\mathcal{R}_N). \]
The minimal $\N$-degree, which is $1$, corresponds to those $f_{\vec{i}}$ with $\vec{i} \in \D_k$.
\begin{align*}
\ld(F_k)(1^K,\boldsymbol{\xi},\boldsymbol{\eta}) & = \sum_{\vec{i} \in \D_k} \sqrt{N}f_{\vec{i}}(1^K,\boldsymbol{\xi},\boldsymbol{\eta})  \\
&= \sum_{h=1}^{\frac{k+1}{2}} \sum_{\vec{i} \in \I_{k,h}} \xi_h \\
&= \sum_{h=1}^{\frac{k+1}{2}}\sigma_{k,h}\xi_h.
\end{align*} 

Now consider $k$ even with $k < K$ even. Then
\[ \sqrt{N}\Delta_{N,k} = \sqrt{N}(D_{N,k} - \E D_{N,k}) = \sum_{\vec{i}\in\Lambda_k}\sqrt{N}(T_{\vec{i}}-\E T_{\vec{i}}). \]
We consider the asymptotics of
\[ \sum_{\vec{i}\in\Lambda_k}\sqrt{N}(T_{\vec{i}}-1),\]
as it is off from $\sqrt{N} \Delta_{N,k}$ by a decreasing constant of order $O(1/\sqrt{N})$. Write
\[ G_k(\mathcal{R}_N) := \sum_{\vec{i} \in \Lambda_k} \sqrt{N}(T_{\vec{i}} - 1) =: \sum_{\vec{i} \in \Lambda_k} g_{\vec{i}}(\mathcal{R}_N). \]
We have that $g_{\vec{i}}$, as a polynomial in $\mathcal{R}_N$, has $\N$-degree at least $2$ if $\vec{i}\not\in\D_k$. On the other hand, for $\vec{i} \in \D_k$, we have
\[ g_{\vec{i}}(\mathcal{R}_N) = \prod_{h=1}^{k/2+1}(1+C_{N,h})^{\ell_{\vec{i}}(h)} - 1. \]
We keep terms of $\N$-degree $1$, so we get
\[\ld(g_{\vec{i}})(\mathcal{R}_N) = \sum_{h=1}^{k/2+1}\ell_{\vec{i}}(h)C_{N,h}.\]
Thus,
\begin{align*}
\ld(G_k)(1^K,\boldsymbol{\xi},\boldsymbol{\eta}) &= \sum_{\vec{i}\in\D_k}\sqrt{N}g_{\vec{i}}(1^K,\boldsymbol{\xi},\boldsymbol{\eta}) \\
&= \sum_{\vec{i}\in\D_k}\sum_{h=1}^{k/2+1}\ell_{\vec{i}}(h)\eta_h \\
&= \sum_{h=1}^{k/2+1}\sigma_{k,h}\eta_h.
\end{align*}

Let us now compute the joint moments. As a notational convenience, we define
\[\zeta_{k,h}=\left\{ \begin{array}{cl} \eta_h & \text{if $k$ even} \\ \xi_h & \text{if $k$ odd.} \end{array} \right.\]
Then, by the above and Section \ref{pre}, we have that
\[\lim_{N\to\infty} \E\prod_{i=1}^\ell (\sqrt{N}\Delta_{N,k_i})^{j_i} = \E\prod_{i=1}^\ell\left(\sum_{h=1}^{\lceil(k_i+1)/2\rceil}\sigma_{k_i,h}\zeta_{k_i,h}\right)^{j_i}.\]
Thus,
\[\sqrt{N}(\Delta_{N,k_i})_{i=1}^\ell \to \left(\sum_{h=1}^{\lceil(k_i+1)/2\rceil}\sigma_{k_i,h}\zeta_{k_i,h}\right)_{i=1}^\ell\]
in distribution as $N\to\infty$. But note that $\zeta=\{\zeta_{0,h}\}\cup\{\zeta_{1,h}\}$ are independent Gaussians, so
\[\left(\sum_{h=1}^{\lceil(k_i+1)/2\rceil}\sigma_{k_i,h}\zeta_{k_i,h}\right)_{i=1}^\ell\]
is jointly Gaussian. Now, it remains to compute the covariances. Let
\[ \mathfrak{D}_k := \sum_{h=1}^{\lceil (k + 1)/2 \rceil} \sigma_{k,h} \zeta_{k,h}. \]
Recall
\[\cov(\zeta_{k_1,h_1},\zeta_{k_2,h_2}) = \frac{2}{\beta}\mathbf{1}_{k_1\equiv k_2\mathrm{mod}2}\mathbf{1}_{h_1=h_2},\]
so it suffices to consider the case where $k_1\equiv k_2\pmod{2}$. The covariance becomes
\[\cov(\mathfrak{D}_{k_1}, \mathfrak{D}_{k_2}) = \frac{2}{\beta}\sum_{h=1}^{\min(r_1,r_2)}\sigma_{k_1,h}\sigma_{k_2,h}\]
where $r_i=\lceil(k_i+1)/2\rceil$. By Lemma \ref{combo}, this is
\[ \cov(\mathfrak{D}_{k_1},\mathfrak{D}_{k_2}) = \frac{4}{\beta}\frac{k_1k_2}{k_1+k_2} \binom{k_1 + k_2 - 2}{\frac{k_1+k_2 -2}{2}}. \]
\end{proof}

\begin{lemma}
\label{combo}
Let $k_1$ and $k_2$ be two positive integers of the same parity and let $r_i=\lceil(k_i+1)/2\rceil$. Then,
\[\sum_{h=1}^{\min(r_1,r_2)}\sigma_{k_1,h}\sigma_{k_2,h} = 2\frac{k_1k_2}{k_1+k_2} \binom{k_1 + k_2 - 2}{\frac{k_1+k_2 -2}{2}} = k_1k_2C_{\frac{k_1+k_2 -2}{2}}.\]
\end{lemma}

\begin{proof}[Proof of Lemma \ref{combo} for $k_1$ odd]
Let $k_1=2k+1$. Note that
\[\sigma_{2k+1,h} = \sum_{\vec{i}\in\D_{2k}}\rho_{\vec{i}}(h)\]
where $\rho_{\vec{i}}(h)$ is the number of times $\vec{i}=(i_1,\ldots,i_{2k})$ hits $h$ and where we count contributions from $i_1$ twice, since there are $\rho_{\vec{i}}(h)$ choices on where to slip in an $(h,h)$ edge into a path from $\D_{2k}$ to get a path in $\I_{2k+1,h}$.

This proof has three key steps. The first step is to relate $\sigma_{2k+1,h}$ to the number of paths in $\D_{2k}$ where the first vertex is fixed to be $h$ using an argument where we consider all rotations of a given path. The second step is to show a bijection from paths in $\D_{2k}$ starting at $h$ to so called ``Catalan paths'', which are paths that start and end at different heights and at each step go up or down by $1$ (this is our $\C_{2k}^h$ in the proof below). The final step is to compute the desired sum, and we use an argument of gluing Catalan paths to form Dyck paths.

Let
\[\D_{2k}^h = \{\vec{i}=(i_1,\ldots,i_{2k})\in\D_{2k} : i_1=h\}\]
and let $\pi_{2k+1,h}=|\D_{2k}^h|$. We first show that $\sigma_{2k+1,h}=(2k+1)\pi_{2k+1,h}$.

Say that $\vec{i},\vec{j} \in \D_{2k}$ are \textit{cyclically equivalent} if $i_\ell = j_{\ell+a}$ for $\ell \in [2k]$ and a fixed constant $a$. The key idea of the proof is to split $\D_{2k}$ into equivalence classes based on cyclic equivalence of paths and find the contribution to $\sigma_{2k+1,h}$ due to a given equivalence class. Let $[\vec{i}]$ be this equivalence class for some $\vec{i}\in\D_{2k}$.

\begin{figure}[ht]
    \centering
    \includegraphics{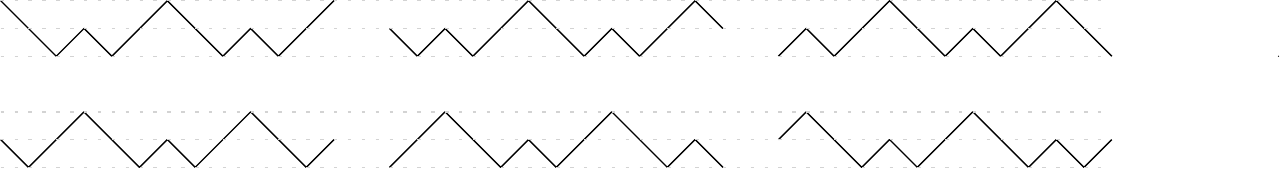}
    \caption{Cyclic equivalence of a path in $\D_{12}$ with period $6$.}
    \label{fig:cyc}
\end{figure}

Let $\vec{i}$ have period $p$, and let it hit $h$ at indices $1\le a_1,\ldots,a_m \le p$. Then, there are exactly $m$ elements of $\D_{2k}^h$ in $[\vec{i}]$, and exactly $2k/p-m$ elements of $[\vec{i}]$ that are not in $\D_{2k}^h$, see Figure \ref{fig:cyc} for an example. Now, if $\vec{j}\in\D_{2k}^h\cap[\vec{i}]$, then $\rho_{\vec{j}}(h)=mp+1$, and if $\vec{j}\in(\D_{2k}\setminus\D_{2k}^h)\cap[\vec{i}]$, then $\rho_{\vec{j}}(h)=mp$. Thus, from these observations, we see that
\[\sum_{\vec{j}\in[\vec{i}]}\rho_{\vec{j}}(h) = m(mp+1)+\left(\frac{2k}{p}-m\right)mp = (2k+1)m,\]
which is miraculously independent of $p$. Thus, each equivalence class contributes $2k+1$ times the number of elements of $\D_{2k}^h$ in it, so the total sum $\sigma_{2k+1,h}$ is $2k+1$ times the total number of elements of $\D_{2k}^h$, or $(2k+1)\pi_{2k+1,h}$.

Define
\[\C_{2k}^h = \{\vec{i}=(i_1,\ldots,i_{2k+1})\in\NN^{2k+1}:|i_j-i_{j+1}|=1\mbox{ and }i_1=1\mbox{ and }i_{2k+1}=2h-1\},\]
where we note that we have $\NN^{2k+1}$ rather than $\NN^{\Z/(2k+1)\Z}$. We claim that in fact, $\pi_{2k+1,h}=|\C_{2k}^h|$, and we show this by providing a bijection $\phi:\D_{2k}^h\to\C_{2k}^h$. Consider some $\vec{i}=(i_1,\ldots,i_{2k})\in\D_{2k}^h$. Let $a$ be the minimal element in $[2k]$ such that $i_a=1$. Then, define
\[\phi(\vec{i}) = (i_a,i_{a+1},\ldots,i_{2k},i_a+h-1,i_{a-1}+h-1,\ldots,i_1+h-1).\]
Now, we provide the inverse map $\psi:\C_{2k}^h\to\D_{2k}^h$. Consider some $\vec{i}=(i_1,\ldots,i_{2k+1})\in\C_{2k}^h$, and let $a$ be the largest element in $[2k+1]$ such that $i_a=h$. Define
\[\psi(\vec{i}) = (i_{2k+1}-h+1,i_{2k}-h+1,\ldots,i_a-h+1=i_1,i_2,\ldots,i_{a-1}).\]
One can easily check that $\phi\circ\psi$ and $\psi\circ\phi$ are both the identity, so $\pi_{2k+1,h}=|\C_{2k}^h|$, see figure \ref{fig:DC}.

\begin{figure}[h]
    \centering
    \includegraphics{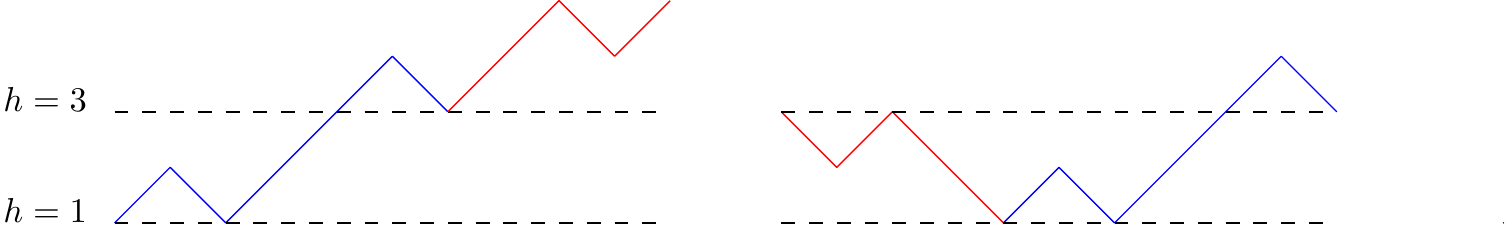}
    \caption{An example mapping between $\C_{10}^3$ and $\D_{10}^3$.}
    \label{fig:DC}
\end{figure}

Finally, we claim that $\displaystyle\bigcup_{h=1}^{\min(r_1,r_2)}\C_{k_1-1}^h\times\C_{k_2-1}^h$ is in bijection with Dyck paths of length $k_1+k_2-2$, as this would finish the proof due to the fact that $\sigma_{k_i,h}=k_i|\C_{k_i-1}^h|$. Consider $(\vec{i},\vec{j})=((i_1,\ldots,i_{k_1}),(j_1,\ldots,j_{k_2}))\in \C_{k_1-1}^h\times\C_{k_2-1}^h$. Consider the map that sends $(\vec{i},\vec{j})$ to the Dyck path $(i_1,\ldots,i_{k_1}=j_{k_2},j_{k_2-1},\ldots,j_1)$ of length $k_1+k_2-2$ (note that length here means number of edges) constructed from $(\vec{i},\vec{j})$. In fact, one can easily check that this map is a bijection from $\displaystyle\bigcup_{h=1}^{\min(r_1,r_2)}\C_{k_1-1}^h\times\C_{k_2-1}^h$ to Dyck paths of length $k_1+k_2-2$, completing the proof.
\end{proof}

\begin{proof}[Proof of Lemma \ref{combo} for $k_1$ even]
The proof will be very similar to that of the $k_1$ odd case. Let $k_1=2k$. Note that
\[\sigma_{2k,h} = \sum_{\vec{i}\in\D_{2k}}\ell_{\vec{i}}(h)\]
where $\ell_{\vec{i}}(h)$ is the number of times that $\vec{i}$ hits $(h,h+1)$. Let
\[\D_{2k}^{h,+} = \{\vec{i}=(i_1,\ldots,i_{2k})\in\D_{2k} : (i_1,i_2)=(h,h+1)\}\]
and let $\pi_{2k,h}=|\D_{2k}^{h,+}|$. We first show that $\sigma_{2k,h}=2k\pi_{2k,h}$. The proof is very similar to the odd case, and define $[\vec{i}]$ in the same way.

Let $\vec{i}$ have period $p$, and let it hit $(h,h+1)$ at indices $1\le a_1,\ldots,a_m\le p$ (i.e. $(i_{a_s},i_{a_s+1})=(h,h+1)$). Then, there are exactly $m$ elements of $\D_{2k}^{h,+}$ in $[\vec{i}]$. Now, if $\vec{j}\in[\vec{i}]$, then $\ell_{\vec{j}}(h)=mp$. Thus, we see that
\[\sum_{\vec{j}\in[\vec{i}]}\ell_{\vec{j}}(h) = \frac{2k}{p}mp = 2km,\]
which is independent of $p$. Thus, each equivalence class contributes $2k$ times the number of elements of $\D_{2k}^{h,+}$ in it, so the total sum $\sigma_{2k,h}$ is $2k$ times the total number of elements of $\D_{2k}^{h,+}$, or $2k\pi_{2k,h}$.

Define
\[\C_{2k-1}^h = \{\vec{i}=(i_1,\ldots,i_{2k})\in\NN^{2k}:|i_j-i_{j+1}|=1\mbox{ and }i_1=1\mbox{ and }i_{2k}=2h\}\]
and
\[\D_{2k-1}^h = \{\vec{i}=(i_1,\ldots,i_{2k})\in\NN^{2k}:|i_j-i_{j+1}|=1\mbox{ and }i_1=h+1\mbox{ and }i_{2k}=h\}.\]
Through effectively the same argument as in the $k$ odd case, we see that $|\C_{2k-1}^h|=|\D_{2k-1}^h|$, and it is obvious that $\D_{2k}^{h,+}$ is in bijection with $\D_{2k-1}^h$, so $|\D_{2k}^{h,+}|=|\C_{2k-1}^h|$. Now, through the exact same argument as in the odd case, we prove that
\[\sum_{h=1}^{\min(r_1,r_2)}\sigma_{k_1,h}\sigma_{k_2,h} = k_1k_2C_{\frac{k_1+k_2 -2}{2}},\]
thus completing the proof of Lemma \ref{combo}.
\end{proof}

\subsection{Proof of Theorem \ref{intvar}}
\begin{proof}
By Theorem \ref{CLT}, we know that $\{\Phi_{N,f}\}_{f \in \mathcal{P}}$ converges to a centered Gaussian family $\{\Phi_f\}_{f \in \mathcal{P}}$. We want to show that
\begin{equation} \label{covPhi}
\cov(\Phi_f,\Phi_g) = \frac{2}{\beta} \int_{-2}^2 f'(x) g'(x) \sigma(x) \, dx.
\end{equation}
Since the monomials $x^k$ form a basis for $\mathcal{P}$, by the bilinearity of covariance and of (\ref{covPhi}) it suffices to show that
\begin{align*}
\cov(\Phi_{x^k}, \Phi_{x^\ell}) &= \frac{2}{\beta} \int_{-2}^2 \frac{d}{dx}(x^k) \frac{d}{dx}(x^\ell) \sigma(x) \, dx \\
&= \frac{2k \ell}{\beta} \int_{-2}^2 x^{k+\ell-2} \sigma(x) \, dx.
\end{align*}
Recalling that the moments of the semicircle law are given by the Catalan numbers, we want to show
\[ \cov(\Phi_{x^k}, \Phi_{x^\ell}) = \mathbf{1}_{2|k+\ell} \frac{4}{\beta} \frac{k\ell}{k+\ell} \binom{k+\ell-2}{\frac{k+\ell-2}{2}} \]
which is given by Theorem \ref{CLT}.
\end{proof}

\section{Identification with Derivative of the Gaussian Free Field} \label{GFF}
In this section, we show that the asymptotic covariance structure can be identified with the derivative of the Gaussian free field.

\subsection{Preservation of Trace Difference}
\label{tridiagonal_discussion}
We demonstrate that the orthogonal conjugation of the GOE matrix into tridiagonal form does not alter the value of the difference of the trace. The same is true for unitary/symplectic conjugation of the GUE/GSE matrix into tridiagonal form. Therefore, the aforementioned results can be thought of as holding for a ``dense'' G$\beta$E matrix. This will be important for us to identify our results to the derivative of the Gaussian Free Field.
  
Recall that the procedure of tridiagonalizing a matrix is a sequence of applications of Householder conjugations. The relevant fact here is that we start with a dense matrix $M$, set $M_0 = M$ and have
\[ M_i = O_i M_{i-1} O_i^t\]
where $O_i$ is an orthogonal matrix of the form
\begin{align*}
O_i = \left[ \begin{array}{cc}
I_i & 0 \\
0 & P_i \end{array} \right]
\end{align*}
where $I_i$ is the $i\times i$ identity matrix and $P_i$ is some orthogonal $(n-i)\times(n-i)$ matrix.
  
Let $D: \R^{n\times n} \to \R$ be the operator
\[ D(M) := \tr M - \tr \widehat{M} \]
where $\widehat{M}$ is the lower principal submatrix. It remains to see that the Householder conjugations do not change the value of $D(M)$. By the structure of the orthogonal matrices, notice that
\[ (M_0)_{11} = (M_1)_{11} = \cdots = (M_{n-1})_{11}.\]
Furthermore, since trace is invariant under orthogonal conjugation, we have
\[ \tr M_0 = \tr M_1 = \cdots = \tr M_{n-1}.\]
Together, these observations give us
\[ DM_0 = \cdots = DM_{n-1}. \]
The relevant properties here indicate that the same argument works for the GUE and GSE, and more generally the G$\beta$E if one considers its heuristic ghost interpretation. For more discussion on the interpretation of G$\beta$E in terms of ``ghosts'' and ``shadows'', see \cite{GHOST}.

\subsection{Review of the Gaussian Free Field} \label{ssec:GFFrev}
Let us begin by recalling the identification with the Gaussian Free Field (GFF) for the Hermite matrices. Let $H_N$ be an $N\times N$ dense GOE matrix if $\beta = 1$, and a $N\times N$ dense GUE matrix if $\beta = 2$ (where we notationally suppress the dependence of the distribution of $H_N$ on $\beta$). Let us also impose the relation that $H_M$ is the principal (say lower right) submatrix of $H_N$ whenever $M \leq N$, for $\beta = 1,2$. The eigenvalues of $\frac{1}{\sqrt{N}} H_{\lfloor vN \rfloor}$ for $v > 0$ concentrate within the interval $(-2 \sqrt{v}, 2\sqrt{v})$. We define the domain on which our eigenvalues concentrate:
\[ \mathbb{D} = \{ (u_1,u_2) \in  \R \times \R_{> 0} : -2 \sqrt{u_2} < u_1 < 2 \sqrt{u_2} \}. \]
Define the map $\Omega: \mathbb{D} \to \mathbb{H}$ where $\mathbb{H}$ is the upper half plane of $\mathbb{C}$
\[ \Omega(\vec{u}) = \frac{u_1}{2} + i \sqrt{ u_2 - \left( \frac{u_1}{2} \right)^2} \]
This map pulls back the conformal structure on $\mathbb{H}$ onto the domain $\mathbb{D}$ of eigenvalues. Define $x(z),y(z)$ by
\[ \Omega^{-1}(z) = (x(z),y(z)) = (2 \Re z, |z|^2). \]
The GFF on $\mathbb{D}$ given by $\Omega$ is the random distribution $\mathfrak{h}$ on $\mathbb{D}$ whose covariance structure is identified by
\[ \cov(\langle \mathfrak{h},\varphi \rangle, \langle \mathfrak{h},\psi\rangle ) = \iint_{\mathbb{D}} \varphi(\vec{u}) \psi(\vec{v}) \left( - \frac{1}{2\pi} \log \left| \frac{\Omega(\vec{u}) - \Omega(\vec{v})}{\Omega(\vec{u}) - \overline{\Omega}(\vec{v})} \right| \right) \, d\vec{u} \, d\vec{v}.\]
Suppose we have a measure $\mu$ supported on a smooth curve $\gamma \subset \mathbb{D}$. Furthermore, let $g(\vec{u})$ be the density of $\mu$ with respect to the natural length measure on $\gamma$. If
\begin{equation} \label{eq:varhg}
\iint_{\gamma \times \gamma} g(\vec{u}) g(\vec{v}) \left( - \frac{1}{2\pi} \log \left| \frac{\Omega(\vec{u}) - \Omega(\vec{v})}{\Omega(\vec{u}) - \overline{\Omega}(\vec{v})} \right| \right) \, d\vec{u} \, d\vec{v} < \infty
\end{equation}
Then
\[ \langle \mathfrak{h}, g \rangle_\gamma := \langle \mathfrak{h},\mu \rangle \]
is a well-defined, centered Gaussian random variable with variance (\ref{eq:varhg}). If $\mu_1,\mu_2$ are measures supported on smooth curves $\gamma_1,\gamma_2$ with densities $g_1,g_2$ respectively, and both satisfy (\ref{eq:varhg}), then we have the covariance
\begin{equation}
\cov( \langle \mathfrak{h}, g_1 \rangle_{\gamma_1}, \langle \mathfrak{h}, g_2 \rangle_{\gamma_2}) = \iint_{\gamma_1 \times \gamma_2} g_1(\vec{u}) g_2(\vec{v}) \left( - \frac{1}{2\pi} \log \left| \frac{\Omega(\vec{u}) - \Omega(\vec{v})}{\Omega(\vec{u}) - \overline{\Omega}(\vec{v})} \right| \right) \, d\vec{u} \, d\vec{v}.
\end{equation}
We finally note that the GFF is conformally invariant. This is a property that will be used later on in the identification of our results with the derivative of the GFF.

The relation between the Hermite matrices and the GFF are given via the height function which is defined to be
\[ h_N(u_1,u_2) = \sqrt{\frac{\beta \pi}{2}} \cdot |\{ \mbox{eigenvalues of $H_{Nu_2}$ that are $< u_1$} \}|. \]

Let $\gamma_u := \Omega^{-1}\{|z|^2 = u\}$. For any $u,v > 0$, \cite{BORODIN} tells us that
\[ \lim_{N\to\infty} \cov \left( \int \! h_N(N^{1/2}x,N u) x^k \, dx, \int \! h_N(N^{1/2}x,N v) x^\ell \, dx \right) =  \cov\left( \langle \mathfrak{h},x^k \rangle_{\gamma_u} , \langle \mathfrak{h}, x^\ell \rangle_{\gamma_v} \right).\]

This is proven by showing that
\begin{equation}
\begin{multlined}
C_{x^k,x^\ell}(u,v) := \lim_{N\to\infty} \cov(N^{-k/2} \tr H_{\lfloor Nu\rfloor}^k, N^{-\ell/2} \tr H_{\lfloor Nv\rfloor}^\ell) \\
= \frac{\beta \pi k \ell}{2} \cov( \langle \mathfrak{h},x^{k-1} \rangle_{\gamma_u} , \langle \mathfrak{h}, x^{\ell-1} \rangle_{\gamma_v} )
\end{multlined}
\end{equation}
and relating the height function with the traces via
\begin{equation} \label{eq:htmat}
\int \! \left[ h_N(N^{1/2}x,N u) - \E\, h_N(N^{1/2}x,N u) \right] x^k \, dx = - \frac{N^{-\frac{k+1}{2}}}{k+1} \sqrt{\frac{\beta\pi}{2}} \left[ \tr H_{\lfloor Nu\rfloor}^{k+1} - \E\, \tr H_{\lfloor Nu\rfloor}^{k+1} \right].
\end{equation}

\subsection{The Derivative of the Gaussian Free Field}
In view of the convergence in Section \ref{ssec:GFFrev}, one may question whether this convergence is robust under differentiation. To phrase this more precisely, define the discrete derivative of the height function
\[ \Delta h_N(x,u) := N\left[ h_N(x,u) - h_N(x,u - N^{-1}) \right] \]
and consider the derivative $\partial_2 \mathfrak{h}(x,u)$ of the GFF $\mathfrak{h}(x,u)$ in the second variable whose covariance is given as follows
\[ \var(\langle \partial_2 \mathfrak{h}, f \rangle_{\gamma_u}) = \lim_{\e \to 0} \var \frac{\langle \mathfrak{h}, f \rangle_{\gamma_u} - \langle \mathfrak{h}, f \rangle_{\gamma_{u-\e}}}{\sqrt{\e}} \]
for smooth functions $f$ on $\gamma_u$. We consider $\partial_2 \mathfrak{h}(x,u)$ as a distribution on $\gamma_u$ for each $u$. The statement that the convergence from Section \ref{ssec:GFFrev} is robust under differentiation is expressed in the following theorem.

\begin{theorem}
For each $u > 0$
\[\lim_{N\to\infty} \cov \left( \langle N^{-1/2} \Delta h_N(Nx,N^{1/2}u), x^k \rangle,\langle N^{-1/2} \Delta h_N(Nx,N^{1/2}u), x^\ell \rangle \right) = \cov \left( \langle \partial_2 \mathfrak{h}, x^k \rangle_{\gamma_u} \langle \partial_2 \mathfrak{h}, x^\ell \rangle_{\gamma_u} \right). \]
\end{theorem}
\begin{proof}
By the conformal invariance of the GFF, we may take $u = 1$. In \cite{ERDOS}, it was shown that
\[ \var\left( \langle \partial_2 \mathfrak{h}, f \rangle_{\gamma_1} \right) = \pi \int_{-2}^2 f(x)^2 \sigma(x) \, dx. \]
In particular
\[ \cov \left(\langle \partial_2 \mathfrak{h}, x^k \rangle_{\gamma_1}, \langle \partial_2 \mathfrak{h}, x^\ell \rangle_{\gamma_1} \right) = \pi \int_{-2}^2 x^{k+\ell} \sigma(x) \, dx. \]
On the other hand, by the relation (\ref{eq:htmat}), we have that
\begin{equation*}
\begin{multlined}
\cov \left( N^{-1/2} \int \Delta h_N(x,u) x^k \, dx, N^{-1/2} \int \Delta h_N(x,u) x^\ell \, dx \right) = \frac{N^{-\frac{k+\ell+2}{2} - 1}}{(k+1)(\ell+1)} \frac{\beta \pi}{2} \\
\times \cov  \left(\tr H_{\lfloor Nu \rfloor}^{k+1} - \tr H_{\lfloor Nu \rfloor - 1}^{k+1}, \tr H_{\lfloor Nu \rfloor}^{\ell+1} - \tr H_{\lfloor Nu \rfloor - 1}^{\ell+1} \right).
\end{multlined}
\end{equation*}
By the discussion in Section \ref{tridiagonal_discussion}, right side of the equality above is equal to
\[ \frac{1}{(k+1)(\ell+1)} \frac{\beta\pi}{2} \cov(\sqrt{N} D_{N,k+1}, \sqrt{N} D_{N,\ell+1}) \to \pi \int_{-2}^2 x^{k+\ell} \sigma(x) \, dx \]
where the convergence follows by Theorem \ref{intvar}.
\end{proof}

\section*{Appendix - Constant Order Spacing}
From Theorem \ref{CLT}, a pair of trace differences $D_{N,k}$, $D_{M,\ell}$ are eventually independent as $M,N$ tend to infinity if $\frac{M-N}{N}$ converges to some nonzero constant. We show that this is no longer the case in general if $M-N$ remains some fixed constant.

Recall the notation from Section \ref{proof_CLT} that $\{\zeta_{k,h}\}$ are a centered Gaussian family with covariance relation
\[\cov(\zeta_{k_1,h_1},\zeta_{k_2,h_2}) = \frac{2}{\beta}\mathbf{1}_{k_1\equiv k_2\mathrm{mod}2}\mathbf{1}_{h_1=h_2}.\]
By very similar arguments as in Section \ref{proof_CLT}, we have the following proposition.

\begin{proposition}
For an integer $r > 0$, constants $C_1,\ldots,C_r$ and nonnegative integers $k_1,\ldots,k_r$, we have
\[\sqrt{N}(D_{N-C_i,k_i})_{i=1}^r\to\left(\sum_{h=1}^{\lceil (k_i+1)/2\rceil}\sigma_{k_i,h}\zeta_{k_i,h+C_i}\right)_{i=1}^r\]
in distribution as $N\to\infty$.
\end{proposition}

Since the limiting random vector is Gaussian, the distribution is determined by the covariance. To obtain the covariance structure, it suffices to compute the covariance for the pair $\displaystyle\left(\sum_{h=1}^{\lceil (k+1)/2 \rceil} \sigma_{k,h} \zeta_{k,h}, \sum_{h=1}^{\lceil (\ell+1)/2 \rceil} \sigma_{\ell,h} \zeta_{\ell,h+C}\right)$. We see that
\begin{align*}
\cov\left(\sum_{h=1}^{\lceil (k+1)/2\rceil}\sigma_{k,h}\zeta_{k,h},\sum_{h=1}^{\lceil (\ell+1)/2\rceil}\sigma_{\ell,h}\zeta_{\ell,h+C}\right)
&= \frac{2}{\beta}\sum_{h\in S}\sigma_{k,h}\,\sigma_{\ell,h-C} \\
&= \frac{2k\ell}{\beta}\sum_{h\in S}\pi_{k,h}\,\pi_{\ell,h-C}
\end{align*}
where $S:=\left[1,\lceil\frac{k+1}{2}\rceil\right]\cap\left[1+C,\lceil\frac{\ell+1}{2}\rceil+C\right]$ and $k\equiv\ell\pmod{2}$ (otherwise the covariance is $0$). From Section \ref{proof_CLT}, note that $\pi_{k,h}=|\C_{k-1}^h|$, so $\pi_{k,h}\pi_{\ell,h-C} = |\C_{k-1}^h\times\C_{\ell-1}^{h-C}|$. We recall that
\[\C_{k-1}^h = \{\vec{i}=(i_1,\ldots,i_k)\in\NN^k : |i_j-i_{j+1}|=1\mbox{ and }i_1=1\mbox{ and }i_k=2h-1+\mathbf{1}_{2\mid k}\}.\]
Note that there is a bijection between $\displaystyle\bigcup_{h\in S}\C_{k-1}^h\times\C_{\ell-1}^{h-C}$ and $\C_{k+\ell-2,C+1}$ given by gluing endpoints of the paths in the following manner
\[((i_1,\ldots,i_k),(j_1,\ldots,j_\ell))\to(i_1,\ldots,i_k=j_\ell+2C,j_{\ell-1}+2C,\ldots,j_1+2C=2C+1),\]
so
\[\sum_{h\in S}\pi_{k,h}\,\pi_{\ell,h-C}=\sum_{h\in S}|\C_{k-1}^h\times\C_{\ell-1}^{h-C}| = \left|\bigcup_{h\in S}\C_{k-1}^h\times\C_{\ell-1}^{h-C}\right|=|\C_{k+\ell-2}^{C+1}|=\pi_{k+\ell-1,C+1}.\]
Thus,
\[\cov\left(\sum_{h=1}^{\lceil (k+1)/2\rceil}\sigma_{k,h}\zeta_{k,h},\sum_{h=1}^{\lceil (\ell+1)/2\rceil}\sigma_{\ell,h}\zeta_{\ell,h+C}\right)=\frac{2k\ell}{\beta}\pi_{k+\ell-1,C+1}.\]
One can check that this is also
\[\frac{2k\ell}{\beta}\left[\binom{k+\ell-2}{\frac{k+\ell}{2}-C-1}-\binom{k+\ell-2}{\frac{k+\ell}{2}-C-2}\right].\]
As a check, in the case $C=0$, we see that
\[\binom{k+\ell-2}{\frac{k+\ell}{2}-1}-\binom{k+\ell-2}{\frac{k+\ell}{2}-2}=C_{\frac{k+\ell}{2}-1},\]
so
\[\cov(\sqrt{N}D_{N,k},\sqrt{N}D_{N,\ell})\to\frac{2k\ell}{\beta}C_{\frac{k+\ell}{2}-1}=\frac{4}{\beta}\frac{k\ell}{k+\ell}\binom{k+\ell-2}{\frac{k+\ell}{2}-1},\]
which is what we showed in Theorem \ref{CLT}.
\pagebreak

\bibliography{mybib}
\bibliographystyle{abbrv}
\end{document}